\documentclass[12pt,oneside,english]{amsart}
\usepackage[T1]{fontenc}
\usepackage[latin9]{inputenc}
\usepackage{geometry}
\usepackage{amstext}
\usepackage{amsthm}
\usepackage{amssymb}
\usepackage{stmaryrd}
\usepackage{setspace}
\makeatletter
\numberwithin{equation}{section}
\numberwithin{figure}{section}
\theoremstyle{plain}
\newtheorem*{conjecture*}{\protect\conjecturename}
\theoremstyle{plain}
\newtheorem{thm}{\protect\theoremname}[section]
\theoremstyle{plain}
\newtheorem{cor}[thm]{\protect\corollaryname}
\theoremstyle{plain}
\newtheorem{prop}[thm]{\protect\propositionname}
\theoremstyle{plain}
\newtheorem{lem}[thm]{\protect\lemmaname}
\theoremstyle{remark}
\newtheorem{rem}[thm]{\protect\remarkname}
\theoremstyle{plain}
\newtheorem{conjecture}[thm]{\protect\conjecturename}
\theoremstyle{plain}
\newtheorem{question}[thm]{\protect\questionname}
\theoremstyle{remark}
\newtheorem*{acknowledgements*}{\protect\acknowledgementname}

\pdfpageattr{/Group << /S /Transparency /I true /CS /DeviceRGB>>}
\usepackage{ae,aecompl} 
\usepackage{enumerate}
\usepackage{mathrsfs}
\usepackage{bbm}
\usepackage{hyperref}
\let\originalleft\left
\let\originalright\right
\renewcommand{\left}{\mathopen{}\mathclose\bgroup\originalleft}
\renewcommand{\right}{\aftergroup\egroup\originalright}

\usepackage{multicol}

\makeatother

\usepackage{babel}
\providecommand{\acknowledgementname}{Acknowledgements}
\providecommand{\conjecturename}{Conjecture}
\providecommand{\corollaryname}{Corollary}
\providecommand{\lemmaname}{Lemma}
\providecommand{\propositionname}{Proposition}
\providecommand{\questionname}{Question}
\providecommand{\remarkname}{Remark}
\providecommand{\theoremname}{Theorem}

\begin{document}
\global\long\def\r#1{#1^{-}}

\global\long\def\br#1#2{{\rm bar}_{#2}\left(#1\right)}

\global\long\def\bry#1{{\rm bar}\left(#1\right)}

\global\long\def\irn{\int_{\R^{n}}}

\global\long\def\d#1{\,{\rm d}#1}

\global\long\def\R{\mathbb{R}}

\global\long\def\C{\mathbb{C}}

\global\long\def\Z{\mathbb{Z}}

\global\long\def\N{\mathbb{N}}

\global\long\def\Q{\mathbb{Q}}

\global\long\def\T{\mathbb{T}}

\global\long\def\F{\mathbb{F}}

\global\long\def\vp{\varphi}

\global\long\def\Sph{\mathbb{S}}

\global\long\def\sub{\subseteq}

\global\long\def\one{\mathbbm1}

\global\long\def\vol#1{\text{vol}\left(#1\right)}

\global\long\def\EE{\mathbb{E}}

\global\long\def\sp{{\rm sp}}

\global\long\def\iprod#1#2{\langle#1,\,#2\rangle}

\global\long\def\uball{B_{2}^{n}}

\global\long\def\conv#1{{\rm conv}\left(#1\right)}

\global\long\def\met#1{{\rm M}\left(#1\right)}

\global\long\def\mfloat#1#2{{\rm M}_{#1}\left(#2\right)}

\global\long\def\mwfloat#1#2#3{{\rm M}_{#1}\left(#2,#3\right)}

\global\long\def\supp#1{{\rm supp}\left(#1\right)}

\global\long\def\eps{\varepsilon}

\global\long\def\PP{\mathbb{P}}

\global\long\def\vein#1{{\rm vein}\left(#1\right)}

\global\long\def\fvein#1{{\rm vein^{*}}\left(#1\right)}

\global\long\def\bfvein#1{{\rm \mathbf{vein^{*}}}\left(\mathbf{#1}\right)}

\global\long\def\ill#1{{\rm ill\left(#1\right)}}

\global\long\def\fill#1{{\rm ill^{*}\left(#1\right)}}

\global\long\def\ovr#1{{\rm ovr\left(#1\right)}}

\global\long\def\norm#1{\left\Vert #1\right\Vert }

\global\long\def\proj{{\rm Pr}}

\global\long\def\ra{\Rightarrow}

\global\long\def\H#1#2{H^{+}\left(#1,\,#2\right)}

\global\long\def\sign{{\rm sign}}

\global\long\def\ent{{\rm Ent}}

\global\long\def\kb#1{\Delta_{KB}\left(#1\right)}

\global\long\def\kn{\mathcal{K}_{n}}

\global\long\def\st{\,:\,}

\title{Improved bounds for Hadwiger's covering problem via thin-shell estimates}
\begin{abstract}
A central problem in discrete geometry, known as Hadwiger's covering
problem, asks what the smallest natural number $N\left(n\right)$
is such that every convex body in $\R^{n}$ can be covered by a union
of the interiors of at most $N\left(n\right)$ of its translates.
Despite continuous efforts, the best general upper bound known for
this number remains as it was more than sixty years ago, of the order
of ${2n \choose n}n\ln n$.

In this note, we improve this bound by a sub-exponential factor. That
is, we prove a bound of the order of ${2n \choose n}e^{-c\sqrt{n}}$
for some universal constant $c>0$. 

Our approach combines ideas from \cite{ArtSlom14-FracCov} by Artstein-Avidan
and the second named author with tools from Asymptotic Geometric Analysis.
One of the key steps is proving a new lower bound for the maximum
volume of the intersection of a convex body $K$ with a translate
of $-K$; in fact, we get the same lower bound for the volume of the
intersection of $K$ and $-K$ when they both have barycenter at the
origin. To do so, we make use of measure concentration, and in particular
of thin-shell estimates for isotropic log-concave measures. 

Using the same ideas, we establish an exponentially better bound for $N\left(n\right)$ when restricting our attention to convex bodies that are $\psi_{2}$. By a slightly different approach, an exponential
improvement is established also for classes of convex bodies with positive modulus of convexity.
\end{abstract}

\author{Han Huang}



\author{Boaz A. Slomka}



\author{Tomasz Tkocz}



\author{Beatrice-Helen Vritsiou}





\thanks{H. Huang: {Department of Mathematics, University of Michigan, 2074 East Hall, 530 Church St., Ann Arbor, MI 48109-1043, USA; e-mail: {sthhan@umich.edu}}}

\thanks{B.A. Slomka: Department of Mathematics, Weizmann Institute of Science, Rehovot 76100, Israel; e-mail: boazslomka@gmail.com (corresponding author)}

\thanks{T. Tkocz: Department of Mathematical Sciences, Carnegie Mellon University, Wean Hall 6113, Pittsburgh, PA 15213, USA; e-mail: ttkocz@math.cmu.edu}

\thanks{B-H. Vritsiou: Department of Mathematical and Statistical Sciences, University of Alberta, CAB 632, Edmonton, AB, Canada T6G 2G1; e-mail: vritsiou@ualberta.ca}

\maketitle

\section{Introduction}

\subsection{Hadwiger's covering problem}

A long-standing problem in discrete geometry asks whether every convex
body in $\R^{n}$ can be covered by a union of at most $2^{n}$ translates
of its interior. It also asks whether $2^{n}$ translates are needed
only for affine images of the $n$-cube. 

This problem was posed by Hadwiger \cite{Hadwiger57} for $n\ge3$
but was already considered and settled for $n=2$ a few years earlier
by Levi \cite{Levi55}. An equivalent formulation, in which the interior
of the convex body is replaced by smaller homothetic copies of it,
was independently posed by Gohberg and Markus \cite{GohMar60}. Other
equivalent formulations of this problem were posed by Hadwiger \cite{Hadwiger60}
and Boltyanski \cite{Bolt60} in terms of illuminating the boundary
of a convex body by outer light sources. For a comprehensive survey
of this problem and most of the progress made so far towards its solution
see e.g. \cite{Bezdek-Khan16,BrassMoser05,Naszodi2018}. 

Putting things formally, a subset of $\R^{n}$ is called a convex
body if it is a compact convex set with non-empty interior. The covering
number of a set $A\sub\R^{n}$ by a set $B\sub\R^{n}$ is given by
\[
N\left(A,B\right)=\min\left\{ N\in\mathbb{N}\,:\,\exists x_{1},\dots,x_{N}\in\R^{n}\text{ such that }A\sub\bigcup_{i=1}^{N}\left\{ x_{i}+B\right\} \right\} ,
\]
where $x+B=\left\{ x+b\,:\,b\in B\right\} $. Denoting the interior
of $B$ by ${\rm int}\,B$ and letting $\lambda B=\left\{ \lambda b\,:\,b\in B\right\} $
for $\lambda\in\R$, Hadwiger's conjecture states the following.
\begin{conjecture*}
Let $K\sub\R^{n}$ be a convex body. Then for some $0<\lambda<1$
one has $N\left(K,\lambda K\right)\le2^{n}$, or equivalently $N\left(K,{\rm int}\,K\right)\le2^{n}$.
Moreover, equality holds only if $K$ is an affine image of the $n$-cube. 
\end{conjecture*}
The currently best general upper bound known for $n\ge3$ is ${2n \choose n}\left(n\ln n+n\ln\ln n+5n\right)$,
while the best bound for centrally-symmetric convex bodies (i.e. convex
bodies $K$ satisfying $K=-K$) is $2^{n}\left(n\ln n+n\ln\ln n+5n\right)$.
Both bounds are simple consequences of Rogers' estimates \cite{Rogers57}
for the asymptotic lower density of a covering of the whole space
by translates of a general convex body, combined with the Rogers-Shephard
inequality \cite{RogShep57}, as can be seen in \cite{Erdos-Rogers64}
and \cite{RogersZong}. For results in small dimensions, see \cite{Bez1, BLNP, BiFo, Bolt1, Dek, IvSt, Las, Pap, PyrmakShepelska18}. We also mention in passing that Hadwiger's conjecture has been confirmed for certain classes of convex bodies such as constant width and fat spindle bodies (see \cite{Bez2, Sch}), belt bodies (see \cite{BoMa1, BoMa2, BoMaHo, Bol2, Ma}), bodies of Helly dimension $2$ (see \cite{Bolt1}), dual cyclic polytopes (see \cite{BeBi, Ta}). We refer to the aforementioned surveys for a detailed account.

A fractional version of the illumination problem was considered by
Nasz\'{o}di \cite{Naszodi09}, where the upper bounds of $2^{n}$ for
the centrally-symmetric case, and ${2n \choose n}$ for the general
case were obtained. The same bounds, as well as the extremity of the
$n$-cube in the centrally-symmetric case, were established by Artstein-Avidan
and the second named author in \cite{ArtSlom14-FracCov} by considering
fractional covering numbers of convex bodies. Moreover, together with
an inequality linking integral covering numbers and fractional covering
numbers (see Section \ref{sec:Hadwiger} below), the aforementioned
best known upper bounds for Hadwiger's classical problem were recovered
(technically, only the bound in the centrally-symmetric case was explicitly
recovered, but the proof of the general bound is almost verbatim the
same). These bounds were recovered once more in \cite{LivshytsTikhomirov16}.
For additional recent results on Hadwiger's problem, see \cite{LivshytsTikhomirov17, Tikhomirov16}, and references therein.

\subsubsection{\textbf{Main results }}

We combine ideas from \cite{ArtSlom14-FracCov} with a new result
on the K\"{o}vner-Besicovitch measure of symmetry for convex bodies, which
we discuss in Section \ref{sec:KB}. As a result, we obtain a new
general upper bound for Hadwiger's problem:
\begin{thm}
\label{thm:Hadwiger}There exist universal constants $c_{1},c_{2}>0$
such that for all $n\ge2$ and every convex body $K\sub\R^{n}$, one
has 
\[
N\left(K,{\rm int}K\right)\le c_{1}4^{n}e^{-c_{2}\sqrt{n}}.
\]
\end{thm}

For $\psi_{2}$ bodies (for definitions and more details see Section
\ref{sec:double-thin-shell} below), we obtain the following exponential
improvement: 
\begin{thm}
\label{thm:Hadwiger-psi2} Let $K\sub\R^{n}$ be a convex body with
barycenter at the origin which is $\psi_{2}$ with constant $b_{2}>0$.
Then 
\[
N\left(K,{\rm int}K\right)\le c_{1}4^{n}e^{-c_{2}b_{2}^{-2}n}.
\]
\end{thm}

\subsection{\label{sec:KB}The K\"{o}vner-Besicovitch measure of symmetry}

Denote the family of all convex bodies in $\R^{n}$ by $\kn$. Denote
the Lebesgue volume of a measurable set $A\sub\R^{n}$ by $\left|A\right|$. 

Let $K\sub\R^{n}$ be a convex body. Given a point $x\in\R^{n}$,
let us call here the set $\left(K-x\right)\cap\left(x-K\right)$ the
\textit{symmetric intersection} of $K$ \textit{at} $x$. As defined
by Gr\"{u}nbaum \cite{Grunbaum63}, the following is a measure of symmetry
for $K$, referred to as the K\"{o}vner-Besicovitch measure of symmetry:
\[
\kb K=\max_{x\in\R^{n}}\frac{\left|\left(K-x\right)\cap\left(x-K\right)\right|}{\left|K\right|}=\max_{x\in\R^{n}}\frac{\left|K\cap\left(x-K\right)\right|}{\left|K\right|}.
\]
To study this quantity, throughout this paper, we use the fact that
the volume of the symmetric intersection of a convex body at a point
$x$ is the same as its convolution square at $2x$, i.e., the convolution
relation 
\[
|(K-x)\cap\left(x-K\right)|=\left|K\cap\left(2x-K\right)\right|=(\one_{K}*\one_{K})\left(2x\right),
\]
where $\one_{K}$ is the indicator function of $K$. Combining this
with the fact that the support of $\one_{K}*\one_{K}$ is $2K$, one
easily obtains by integration that 
\begin{equation}
\min_{K\in\kn}\kb K\ge2^{-n}.\label{eq:KB_bound}
\end{equation}

Denote by ${\rm b}(K)$ the barycenter of $K.$ By fixing this as
the point of reference, one may consider the volume ratio of the symmetric
intersection of $K$ at its barycenter as another measure of symmetry
for $K$. A result of V. Milman and Pajor \cite{PajorMilman2000}
tells us that 
\begin{equation}
\frac{\left|(K-{\rm b}(K))\cap({\rm b}(K)-K)\right|}{\left|K\right|}\ge2^{-n}.\label{eq:MP}
\end{equation}

The optimal lower bound, in both instances, is not known and conjectured
to be attained by the simplex, which would imply a lower bound of
the order of $\left(\frac{2}{e}\right)^{n}$ (see e.g. \cite{Grunbaum63},
\cite{Taschuk15} for more details).

\subsubsection{\textbf{A new lower bound}}

Our second goal in this note is to improve both (\ref{eq:KB_bound})
and (\ref{eq:MP}). We consider two approaches, both of which involve
using the property of a (properly normalized) log-concave measure
to concentrate in a thin-shell, and in particular a quantitative form
of it by Gu\'{e}don and E. Milman \cite{GuedonMilman2011}. More precisely,
let $X$ and $Y$ be independent random vectors, uniformly distributed
on a convex body $K\sub\R^{n}$. Our first approach is based on the
comparison of the measure of a ball, whose boundary is between the
two thin shells around which the distributions of $X$ and $\text{\ensuremath{\left(X+Y\right)}/2}$
are concentrated, according to each of these measures; this leads
to the improvement of (\ref{eq:KB_bound}).

The second approach, which allows us to bound the volume of the symmetric
intersection of $K$ at its barycenter and to improve (\ref{eq:MP}),
combines the above mentioned thin-shell estimates of Gu\'{e}don and E.
Milman with the notion of entropy. Given that there is not much reason
to believe our bounds are optimal, we have chosen to present both
approaches since either might have the potential to give further improvements.

To turn to details, we prove the following:
\begin{thm}
\label{thm:main}For some universal constant $c>0$, we have 
\[
\min_{K\in\kn}\kb K\geq\min_{K\in\kn\,:\,{\rm b}(K)=0}\frac{\left|K\cap(-K)\right|}{|K|}\ge\frac{\exp\left(cn^{1/2}\right)}{2^{n}}.
\]
\end{thm}

Theorem \ref{thm:main} is a particular consequence of Propositions
\ref{prop:KB_psi} and \ref{prop:MP_psi} below, which provide a lower
bound for $\kb K$ and $|K\cap(-K)|/|K|$ by taking into account the
$\psi_{\alpha}$ behavior of the convex body $K$ (for definitions
and more details see Section \ref{sec:double-thin-shell} below).
In particular, for $\psi_{2}$ bodies, we have the following exponential
improvement of (\ref{eq:KB_bound}) and \ref{eq:MP}.
\begin{cor}
\label{cor:psi_2}(of Propositions \ref{prop:KB_psi} and \ref{prop:MP_psi})
Let $K\in\R^{n}$ be a convex body centered at the origin which is
$\psi_{2}$ with constant $b_{2}>0$. Then 
\[
\kb K\geq\frac{\left|K\cap(-K)\right|}{|K|}\ge\frac{\exp\left(cb_{2}^{-2}n\right)}{2^{n}}.
\]
\end{cor}

\subsection{Positive modulus of convexity }

The modulus of convexity of a centered convex body $K\sub\R^{n}$
is defined by
\[
\delta_{K}\left(\eps\right)=\inf\left\{ 1-\norm{\frac{x+y}{2}}_{K}\st\norm x_{K},\norm y_{K}\leq1,\,\norm{x-y}_{K}\ge\eps\right\} ,
\]
where $\norm x_{K}=\inf\left\{ r>0\st x\in rK\right\} $ is the gauge
function of $K$. We say that $K$ is uniformly convex if $\delta_{K}\left(\eps\right)>0$
for all $0<\eps<2$. Note that in the finite--dimensional case, $K\sub\R^{n}$
is strictly convex (i.e. the boundary of $K$ contains no line segments)
if and only if it is uniformly convex.

Using a different concentration result of Arias-De-Reyna, Ball, and
Villa \cite{arias-de-reyna_ball_villa98}, which was generalized by
Gluskin and Milman \cite{GluskinMilman2004}, we extend Theorems \ref{thm:Hadwiger-psi2}
and \ref{cor:psi_2} to the class of convex bodies whose modulus of
convexity is positive for some $0<\eps<\sqrt{2}$. More precisely,
for $0<r<1$ and $0<\eps<\sqrt{2}$, let $\mathcal{K}_{n,r,\eps}$
be the class of centered convex bodies $K\sub\R^{n}$ for which $\delta_{K}\left(\eps\right)\ge r$. 
\begin{thm}
\label{thm:uniform_convex}Let $0<r<1$, $0<\eps<\sqrt{2}$, and let
$K\in\mathcal{K}_{n,r,\eps}$. Then, for $\alpha:=1-\exp\left(-\frac{\left(\sqrt{2}-\eps\right)^{2}}{4}n\right)$,
we have

\begin{gather*}
\kb K\geq\alpha\,2^{-n}\left(\frac{1}{1-r}\right)^{n},\quad\frac{\left|K\cap(-K)\right|}{|K|}\ge\frac{1}{e\sqrt{n}}\,2^{-n}\left(\frac{1}{1-\alpha r}\right)^{n}.
\end{gather*}
\end{thm}

\begin{thm}
\label{thm:Hadwiger_UnifConv}Let $0<r<1$, $0<\eps<\sqrt{2}$, and
let $K\in\mathcal{K}_{n,r,\eps}$. Then 
\[
N\left(K,{\rm int}K\right)\le\left(1-e^{-\frac{\left(\sqrt{2}-\eps\right)^{2}}{4}n}\right)^{-1}\left(4\left(1-r\right)\right)^{n}.
\]
 
\end{thm}

The paper is organised as follows. In Section \ref{sec:double-thin-shell}
we prove the first part of Theorem \ref{thm:main} and of Corollary
\ref{cor:psi_2} (the bounds for the K\"{o}vner-Besicovitch measure of
symmetry), and in Section \ref{sec:Hadwiger} we apply these to Hadwiger's
covering problem. Section \ref{sec:Unif_conv} is devoted to the respective
bounds in the case of uniformly convex bodies, i.e. the first part
of Theorem \ref{thm:uniform_convex} as well as Theorem \ref{thm:Hadwiger_UnifConv}.
Finally, in Section \ref{sec:entropy} we complete the proofs of Theorems
\ref{thm:main} and \ref{thm:uniform_convex} and of Corollary \ref{cor:psi_2}
by showing via our second approach how to bound the volume of the
symmetric intersection of $K$ at its barycenter as well. A couple
of concluding remarks are gathered at the end, including an application to a conjecture by Ehrhart in the geometry of numbers.

\section{\label{sec:double-thin-shell}Bounding the convolution square }

This section is devoted to the proof of Proposition \ref{prop:KB_psi}
below. To that end, we need to recall some facts and results.

Denote the standard Euclidean inner product on $\R^{n}$ by $\iprod{\cdot}{\cdot}$,
and the corresponding Euclidean norm on $\R^{n}$ by $\norm{\cdot}_{2}$.
We shall also denote probability by $\PP$ and expectation by $\EE$. 

Recall that a random vector in $\R^{n}$ is called isotropic if $\EE X=0$
(i.e., its barycenter is the origin) and $\EE\left(X\otimes X\right)=Id$
(i.e., its covariance matrix is the identity). We say that $X$ is
$\psi_{\alpha}$ with constant $b_{\alpha}$ if 
\[
\left(\EE\left|\iprod Xy\right|^{p}\right)^{1/p}\le b_{\alpha}p^{1/\alpha}\left(\EE\left|\iprod Xy\right|^{2}\right)^{1/2}\,\,\,\forall p\ge2,\,\,\forall y\in\R^{n}.
\]
A function $f:\R^{n}\to\left[0,\infty\right)$ is called log-concave
if $\log f$ is concave on the support of $f$. It is well-known that
any random vector $X$ in $\R^{n}$ with a log-concave density is
$\psi_{1}$ with $b_{1}\le C$, for some universal constant $C>0$
(see e.g. \cite[p. 115]{AGM15}). 

We shall need the following thin-shell deviation estimate of Gu\'{e}don
and E. Milman:
\begin{thm}
\label{thm:GM2011}(\cite[Theorem 1.1]{GuedonMilman2011}) Let $X$
denote an isotropic random vector in $\R^{n}$ with log-concave density,
which is in addition $\psi_{\alpha}$~($\alpha\in\text{\ensuremath{\left[1,2\right]}})$
with constant $b_{\alpha}$. Then, 
\[
\PP\left(\left|\norm X_{2}-\sqrt{n}\right|\ge t\sqrt{n}\right)\le C\exp\left(-c'b_{\alpha}^{-\alpha}\min\left(t^{2+\alpha},t\right)n^{\alpha/2}\right)\quad\forall t\ge0,
\]
where $c'>0$ is some universal constant. 
\end{thm}

We remark that the dependence in $n$ in Theorem \ref{thm:GM2011}
is optimal, while the dependence in $t$ was recently improved by
Lee and Vempala \cite{LeeVempala18} in the $\psi_{1}$ case. However,
in our approach $t$ is going to be some fixed number which is bounded
away from $0$, thus optimizing over it cannot yield better bounds. 
\begin{prop}
\label{prop:KB_psi}Suppose $K$ is a convex body centered at the
origin which is $\psi_{\alpha}$ with constant $b_{\alpha}$. Then,
for some universal constant $c>0$,
\[
\kb K\ge\frac{\exp\left(cb_{\alpha}^{-\alpha}n^{\alpha/2}\right)}{2^{n}}.
\]
\end{prop}

We remark that Theorem \ref{thm:main} is a particular consequence
of Proposition \ref{prop:KB_psi}, as all random vectors with log-concave
densities are $\psi_{1}$ with the same universal constant. 
\begin{proof}[Proof of Proposition \ref{prop:KB_psi}]
Let $X$ and $Y$ be  independent random vectors, uniformly distributed
on $K$. Since $\kb K$ is affine invariant, we may assume without
loss of generality that $K$ is in isotropic position: this means
that $\left|K\right|=1,$ ${\rm b}(K)$ = 0 as assumed already, and
that $\EE\left(X\otimes X\right)$ is a multiple of the identity,
$\EE\left(X\otimes X\right)=L_{K}^{2}Id$ where $L_{K}$ is called
the isotropic constant of $K$ (note that this is another well-defined
affine invariant of $K$). Equivalently, we ask that $\left|K\right|=1$
and $X/L_{K}$ is isotropic as defined above.

We are now looking for a lower bound for $\norm f_{\infty}$ where
$f=\one_{K}*\one_{K}$ is the density function for the random vector
$X+Y$. Instead, we shall work with $\frac{X+Y}{2}$ so that both
$\frac{X+Y}{2}$ and $X$ have the same support. The probability density
function of $\frac{X+Y}{2}$ is then $g\left(x\right)=f\left(2x\right)2^{n}$.
There are many nice properties that $\frac{X+Y}{2}$ inherits from
$X$. In particular, $\frac{X+Y}{2}$ has a centered log-concave density
(the latter is a consequence of the Pr\'{e}kopa-Leindler inequality, see
e.g. \cite{AGM15}). Moreover, 
\begin{align*}
 & \text{\ensuremath{\mathbb{\EE}}}_{X,Y}\left(\frac{X+Y}{2}\right)\otimes\left(\frac{X+Y}{2}\right)\\
= & \frac{1}{4}\EE_{X,Y}\left(X\otimes X+X\otimes Y+Y\otimes X+Y\otimes Y\right)\\
= & \frac{1}{4}\left(L_{K}^{2}I+0+0+L_{K}^{2}I\right)\\
= & \frac{1}{2}L_{K}^{2}I.
\end{align*}
Thus, $\frac{X+Y}{2}$ is isotropic up to scaling. Finally, $\frac{X+Y}{2}$
has more or less the same $\psi_{\alpha}$ behavior as $X$ (indeed,
the above computations already show that
\[
\left(\EE\left|\iprod{\tfrac{X+Y}{2}}y\right|^{2}\right)^{1/2}=\,\frac{1}{\sqrt{2}}L_{K}\|y\|_{2}\,=\,\frac{1}{\sqrt{2}}\left(\EE\left|\iprod Xy\right|^{2}\right)^{1/2}
\]
 for every $y\in\R^{n},$ hence a single application of Minkowski's
inequality gives

\begin{align*}
\left(\EE\left|\iprod{\tfrac{X+Y}{2}}y\right|^{p}\right)^{1/p}\leq2\left(\EE\left|\iprod{\tfrac{X}{2}}y\right|^{p}\right)^{1/p} & =\left(\EE\left|\iprod Xy\right|^{p}\right)^{1/p}\\
 & \le b_{\alpha}p^{1/\alpha}\left(\EE\left|\iprod Xy\right|^{2}\right)^{1/2}\\
 &=\sqrt{2}\,b_{\alpha}p^{1/\alpha}\left(\EE\left|\iprod{\tfrac{X+Y}{2}}y\right|^{2}\right)^{1/2},
\end{align*}
assuming $X$ is $\psi_{\alpha}$ with constant $b_{\alpha}$. It
is worth remarking however that, for our proof here, the fact that
the distribution of $\frac{X+Y}{2}$ is $\psi_{1}$ suffices (and,
as mentioned already, this is true for every log-concave distribution).

Observe now that for any $r>0$ we have
\[
\norm g_{\infty}\ge\frac{\int_{rL_{K}\sqrt{n}B_{2}^{n}\cap K}g\left(x\right)\d x}{\int_{rL_{K}\sqrt{n}B_{2}^{n}\cap K}1\d x}=\frac{\PP\left(\norm{\frac{X+Y}{2}}_{2}\le rL_{K}\sqrt{n}\right)}{\PP\left(\norm X_{2}\le rL_{K}\sqrt{n}\right)}.
\]
Since $\EE_{X,Y}\norm{\frac{X+Y}{2}}_{2}^{2}=\frac{1}{2}nL_{K}^{2}$
and $\EE_{X}\norm X_{2}^{2}=nL_{K}^{2}$, we know that the distributions
of $X$ and $\frac{X+Y}{2}$ are concentrated within two different
thin-shells. Thus, for $\frac{1}{\sqrt{2}}<r<1$, we get that $\PP_{X,Y}\left(\norm{\frac{X+Y}{2}}_{2}\le rL_{K}\sqrt{n}\right)$
is almost $1$ since the set considered includes the ``good'' thin-shell
of $\frac{X+Y}{2}$. On the other hand, $\PP\left(\norm X_{2}\le rL_{K}\sqrt{n}\right)$
is almost $0$ since the set considered excludes the corresponding
thin-shell of $X$. To quantify this, we apply Theorem \ref{thm:GM2011}:
for any isotropic $\psi_{\alpha}$ log-concave vector $Z$ the inequality
in \ref{thm:GM2011} is split into 
\[
\PP\left(\norm Z_{2}\le\left(1-t\right)\sqrt{n}\right)\le C\exp\left(-c'b_{\alpha}^{-\alpha}\min\left(t^{2+\alpha},t\right)\sqrt{{n}}\right)\quad\forall t\in[0,1],
\]
\[
\PP\left(\norm Z_{2}\ge\left(1+t\right)\sqrt{n}\right)\le C\exp\left(-c'b_{\alpha}^{-\alpha}\min\left(t^{2+\alpha},t\right)\sqrt{n}\right)\quad\forall t\ge0.
\]
Since we shall apply the first one with $Z$ replaced by $X/L_{K}$
and the second one with $Z$ replaced by $\frac{X+Y}{2}\cdot\frac{\sqrt{2}}{L_{K}}$,
we need $1-t=\frac{1+t}{\sqrt{2}}$ and hence $t=\frac{\sqrt{2}-1}{\sqrt{2}+1}$.
We thus obtain
\[
\PP\left(\norm X_{2}\le\frac{2}{\sqrt{2}+1}L_{K}\sqrt{n}\right)\le\exp\left(-c'b_{\alpha}^{-\alpha}n^{\alpha/2}\right),
\]
and 
\[
\PP\left(\norm{\frac{X+Y}{2}}_{2}\le\frac{2}{\sqrt{2}+1}L_{K}\sqrt{n}\right)\ge1-\exp\left(-c'b_{\alpha}^{-\alpha}n^{\alpha/2}\right).
\]
Therefore, we conclude that for some universal constant $c>0$
\[
\norm g_{\infty}\ge\exp\left(cb_{\alpha}^{-\alpha}n^{\alpha/2}\right),
\]
and equivalently
\[
\kb K=\frac{\norm g_{\infty}}{2^{n}}\ge\frac{\exp\left(cb_{\alpha}^{-\alpha}n^{\alpha/2}\right)}{2^{n}}.
\]
\end{proof}

\section{\label{sec:Hadwiger}A new bound for Hadwiger's covering problem}

This section is devoted to the proof of Theorems \ref{thm:Hadwiger}
and \ref{thm:Hadwiger-psi2}. To that end, we need some preliminaries. 

Let $\overline{N}\left(A,B\right)=\min\left\{ N\,:\,\exists x_{1},\dots,x_{N}\in A\text{ such that }A\sub\bigcup_{i=1}^{N}\left\{ x_{i}+B\right\} \right\} $
be the covering number of $A$ by translates of $B$ that are centered
in $A$. We shall need the following volume ratio bound. 
\begin{lem}
\label{fact:vol_bd}Let $A,B\sub\R^{n}$ be convex bodies. Suppose
$B$ contains the origin in its interior. Then 
\[
\overline{N}\left(A,B\right)\le2^{n}\frac{\left|A+\frac{1}{2}(B\cap(-B))\right|}{\left|B\cap(-B)\right|}.
\]
\end{lem}

\begin{proof}
Recall that the separation number of $A$ in $B$ is defined as 
\[
M\left(A,B\right)=\max\left\{ M\,:\,\exists x_{1},\dots,x_{M}\in A\text{ such that }\forall i\neq j\,\left(x_{i}+B\right)\cap\left(x_{j}+B\right)=\emptyset\right\} .
\]
It is an easy exercise (see e.g. \cite{ArtSlom14-FracCov}) to show
that 
\[
M\left(A,B\right)\le\frac{\left|A+B\right|}{\left|B\right|}.
\]
Next, note that for any convex body $T\sub\R^{n}$, one has $\overline{N}\left(A,T-T\right)\le M\left(A,T\right).$
Indeed, take a maximal $T$-separated set in $A$, that is a set of
points $x_{1},\dots,x_{M}\in A$ such that for every point $x\in A$
one has $\left(x+T\right)\cap\bigcup_{i=1}^{M}\left\{ x_{i}+T\right\} \neq\emptyset$.
This means that $A\sub\bigcup_{i=1}^{M}\left\{ x_{i}+T-T\right\}$ or, in other words, that $\overline{N}\left(A,T-T\right)\le M\left(A,T\right)$.
Since $\overline{N}\left(A,B\right)\le\overline{N}\left(A,B\cap(-B)\right)$, it follows that $\overline{N}\left(A,B\right)\le M\left(A,\frac{1}{2}\left(B\cap(-B)\right)\right)$,
and hence 
\[
\overline{N}\left(A,B\right)\le2^{n}\frac{\left|A+\frac{1}{2}\left(B\cap(-B)\right)\right|}{\left|B\cap(-B)\right|}.
\]
\end{proof}
Next, we recall the notion of fractional covering numbers, as defined in \cite{ArtSlom14-FracCov}. Remember that $\one_{A}$ stands for the indicator function of a set $A\sub\R^{n}$. A sequence of pairs of points and weights $S=\{(x_{i}\,,\,\omega_{i}):\,x_{i}\in\R^{n},\,\omega_{i}\in\R^{+}\}_{i=1}^{N}$ is said to be a fractional covering of a set
$K\sub\R^{n}$ by a set $T\sub\R^{n}$ if for all $x\in K$ we have
$\sum_{i=1}^{N}\omega_{i}\one_{x_{i}+T}\left(x\right)\ge1$. The total
weight of the covering is denoted by $\omega(S)=\sum_{i=1}^{N}\omega_{i}$.
The fractional covering number of $K$ by $T$ is defined to be the
infimal total weight over all fractional coverings of $K$ by $T$
and is denoted by $N_{\omega}(K,T)$. 

We shall also need the following volume ratio bound from \cite{ArtSlom14-FracCov}:
\begin{lem}[{\cite[Proposition 2.9]{ArtSlom14-FracCov}}]
\label{lem:vol_Nw}Let $K,T\sub\R^{n}$ be convex bodies. Then 
\[
N_{\omega}\left(K,T\right)\le\frac{\left|K-T\right|}{\left|T\right|}.
\]
\end{lem}

Finally, we shall need the following inequality that relates integral covering numbers and fractional covering numbers, and which was proved
in \cite{Naszodi2014}, improving  on a similar inequality in \cite{ArtSlom14-FracCov}.  For any bounded Borel measurable sets, $K,T_1$ and $T_2$, one has
\begin{equation}
N\text{\ensuremath{\left(K,T_{1}+T_{2}\right)\le N_{\omega}\left(K,T_{1}\right)\left(1+\ln\overline{N}\left(K,T_{2}\right)\right)}}.\label{eq:N_vs_N*}
\end{equation}

To be more precise, \eqref{eq:N_vs_N*} immediately follows from  \cite[Theorem 1.2]{Naszodi2014}, applied with $L=T_1+T_2$ and $T=T_2$.

\begin{proof}[Proof of Theorem \ref{thm:Hadwiger} ]
We can assume without loss of generality that ${\rm b}(K)=0.$ By
Lemma \ref{lem:vol_Nw}, for $0<\lambda<1$ and any $x\in\R^{n}$
we have 
\begin{align*}
N_{\omega}\left(K,\lambda K\right) & \le N_{\omega}\left(K,\lambda\left(K\cap\left(x-K\right)\right)\right)\le\frac{\left|K-\lambda\left(K\cap\left(x-K\right)\right)\right|}{\left|\lambda\left(K\cap\left(x-K\right)\right)\right|}\\
 & \le\left(\frac{1+\lambda}{\lambda}\right)^{n}\frac{\left|K\right|}{\left|K\cap\left(x-K\right)\right|}.
\end{align*}
By applying Theorem \ref{thm:main} with the point $x$ which maximizes
the above volume ratio, we get 
\[
N_{\omega}\left(K,\lambda K\right)\le\left(\frac{1+\lambda}{\lambda}\right)^{n}2^{n}e^{-c\sqrt{n}}.
\]
Using (\ref{eq:N_vs_N*}) with $T_{1}=\alpha\lambda K,T_{2}=\left(1-\alpha\right)\lambda K$
for some $\alpha\in\left(0,1\right)$, we obtain
\[
N\left(K,\lambda K\right)\le\left(\frac{1+\alpha\lambda}{\alpha\lambda}\right)^{n}2^{n}e^{-c\sqrt{n}}\left(1+\ln\overline{N}\left(K,\left(1-\alpha\right)\lambda K\right)\right).
\]
Using Lemma \ref{fact:vol_bd} and taking the limit $\lambda\nnearrow1$,
we get
\begin{align*}
N\left(K,{\rm int}K\right) & \le\left(\frac{1+\alpha}{\alpha}\right)^{n}2^{n}e^{-c\sqrt{n}}\left(1+\ln\left(2^{n}\frac{\left|K+\frac{1}{2}\left(1-\alpha\right)\left(K\cap(-K)\right)\right|}{\left|\left(1-\alpha\right)\left(K\cap(-K)\right)\right|}\right)\right)\\
 & \le\left(\frac{1+\alpha}{\alpha}\right)^{n}2^{n}e^{-c\sqrt{n}}\left(1+\ln\left(\left(\frac{4}{1-\alpha}\right)^{n}\frac{\left|K\right|}{\left|K\cap(-K)\right|}\right)\right).
\end{align*}
Since $K$ is centered at the origin, (\ref{eq:MP}) (or its improvement
in Theorem \ref{thm:main}, which however cannot essentially affect
the final estimate here) implies that 
\begin{align*}
N\left(K,{\rm int}K\right) & \le\left(\frac{1+\alpha}{\alpha}\right)^{n}2^{n}e^{-c\sqrt{n}}\left(1+\ln\left(\left(\frac{4}{1-\alpha}\right)^{n}2^{n}\right)\right)\\
 & \le\left(\frac{1+\alpha}{\alpha}\right)^{n}2^{n}e^{-c\sqrt{n}}\left(1+n\ln\left(\frac{8}{1-\alpha}\right)\right).
\end{align*}
Plugging in $\alpha=1-1/n$ yields that, for some universal constants
$c_{1},c_{2}>0$, we have
\begin{align*}
N\left(K,{\rm int}K\right) & \le c_{1}4^{n}e^{-c_{2}\sqrt{n}}\text{.}
\end{align*}
\end{proof}
The proof of Theorem \ref{thm:Hadwiger-psi2} is the same as that of Theorem \ref{thm:Hadwiger}, except that one uses Corollary \ref{cor:psi_2} instead of Theorem \ref{thm:main}.
\section{\label{sec:Unif_conv}Positive modulus of convexity }
Recall that the modulus of convexity of a centered convex body $K\sub\R^{n}$
is defined by 
\[
\delta_{K}\left(\eps\right)=\inf\left\{ 1-\norm{\frac{x+y}{2}}_{K}\st\norm x_{K},\norm y_{K}\le1,\,\norm{x-y}_{K}\ge\eps\right\} .
\]

A result of Arias-De-Reyna, Ball, and Villa \cite{arias-de-reyna_ball_villa98},
which was generalized by Gluskin and Milman \cite{GluskinMilman2004},
tells us that if $K\sub\R^{n}$ is a convex body such that $0\in{\rm int}\,K$
and $\left|K\right|=1$  then for all $0<\eps'<1$ one has 
\begin{equation}
\left|\left\{ \left(x,y\right)\in K\times K\st\norm{x-y}_{K}\le\sqrt{2}\left(1-\eps'\right)\right\} \right|\le e^{-\eps'^{2}n/2}.\label{eq:conc_dist}
\end{equation}

We use this result to prove Theorem \ref{thm:uniform_convex}: 
\begin{proof}[Proof of first part of Theorem \ref{thm:uniform_convex}]
Without loss of generality, we assume that $\left|K\right|=1$. Let
$X$ and $Y$ be independent random vectors, uniformly distributed
on $K$. Let $f\left(x\right)=\left|K\cap\left(x-K\right)\right|$
and recall that the density of $\frac{X+Y}{2}$ is $g\left(x\right)=2^{n}f\left(2x\right)$. 

Since, by assumption, $\delta_{K}\left(\eps\right)\ge r$, the set
$\theta=\left\{ \left(x,y\right)\in K\times K\st\norm{x-y}_{K}\geq\eps\right\} $
satisfies that 
\[
\theta\sub\left\{ \left(x,y\right)\in K\times K\st\frac{x+y}{2}\in\left(1-r\right)K\right\} .
\]
By (\ref{eq:conc_dist}), we have that $\left|\theta\right|\ge1-e^{-\frac{\left(\sqrt{2}-\eps\right)^{2}n}{4}}$
and hence 
\begin{align*}
\PP\left(\frac{X+Y}{2}\in\left(1-r\right)K\right) & =\iint_{\left\{ \left(x,y\right)\in K\times K\st\frac{x+y}{2}\in\left(1-r\right)K\right\} }\d x\d y\\
 & \ge\iint_{\theta}\d x\d y\ge1-e^{-\frac{\left(\sqrt{2}-\eps\right)^{2}n}{4}}.
\end{align*}
 Therefore, it follows that 
\begin{align*}
\norm g_{\infty} & \ge\frac{\int_{\left(1-r\right)K}g\left(x\right)\d x}{\int_{\left(1-r\right)K}\d x}=\frac{\PP\left(\frac{X+Y}{2}\in\left(1-r\right)K\right)}{\PP\left(X\in\left(1-r\right)K\right)}\\
 & \ge\left(\frac{1}{1-r}\right)^{n}\left(1-e^{-\frac{\left(\sqrt{2}-\eps\right)^{2}n}{4}}\right).
\end{align*}
\end{proof}
Repeating the proof of Theorem \ref{thm:Hadwiger} but now using Theorem
\ref{thm:uniform_convex}, Theorem \ref{thm:Hadwiger_UnifConv} follows.\\

\section{\label{sec:entropy}Bounding the convolution square at the barycenter}

This section is devoted to the proof of Proposition \ref{prop:MP_psi}
below (which will give the full proofs of Theorem \ref{thm:main}
and Corollary \ref{cor:psi_2}) as well as completing that of Theorem
\ref{thm:uniform_convex} (the arguments will be very similar, just
different applications of the same method). We recall that, for a
random vector $X$ in $\R^{n}$ with density $f,$ we define its entropy
as 

\[
\ent[X]=-\int_{\R^{n}}f\ln f.
\]

The conclusions of the following standard lemma are simple consequences
of Jensen's inequality.
\begin{lem}
\label{lem:jensen}For any measurable function $h:\R^{n}\to[0,+\infty)$
which is positive on the support of $f$ we have

\begin{equation}
\ent[X]\leq-\int_{\R^{n}}f\ln h\,+\,\ln\left(\int_{\R^{n}}h\right),\label{eq:lem-jensen1}
\end{equation}
assuming all the quantities are finite. Moreover, if $X$ has a log-concave
density, then 

\begin{equation}
\ent[X]=\EE[-\ln f(X)]\geq-\ln f(\EE X).\label{eq:lem-jensen2}
\end{equation}
\end{lem}

\begin{proof}
To prove (\ref{eq:lem-jensen1}), we write

\[
\ent[X]+\int_{\R^{n}}f\ln h=\int_{\R^{n}}f\ln\frac{h}{f}\leq\ln\left(\int_{\R^{n}}h\right),
\]
with the inequality following by Jensen's inequality. As for (\ref{eq:lem-jensen2}),
we note that, if $f$ is assumed log-concave, $-\ln f$ will be a
convex function on $\R^{n},$ which allows to apply Jensen's inequality
again. 
\end{proof}
\begin{rem}
We will apply Lemma \ref{lem:jensen} as follows. If $K\subset\R^{n}$
is a centered convex body, and $X,Y$ are independent random vectors
uniformly distributed on $K,$ then the density $f$ of $X$ is given
by $f(x)=\frac{1}{|K|}\one_{K}$, while the density $g$ of $X+Y$
by $g(x)=\frac{1}{|K|^{2}}(\one_{K}*\one_{K})\left(x\right)=\frac{1}{|K|^{2}}|K\cap(x-K)|$
(recall that $X+Y$ has a centered log-concave density, which is not
hard to check using this identity). These show that $\ent[X]=\ln|K|,$
while, by (\ref{eq:lem-jensen2}),
\[
-\ln\left(\frac{|K\cap(-K)|}{|K|^{2}}\right)=-\ln g(0)\leq\ent[X+Y].
\]

Therefore,

\begin{equation}
-\ln\left(\frac{|K\cap(-K)|}{|K|}\right)=-\ln\left(\frac{|K\cap(-K)|}{|K|^{2}}\right)-\ln|K|\leq\ent[X+Y]-\ent[X],\label{eq:ent-app1}
\end{equation}
 which we can combine with (\ref{eq:lem-jensen1}), applied for the
vector $X+Y,$ to obtain

\begin{equation}
-\ln\left(\frac{|K\cap(-K)|}{|K|}\right)\leq\EE[-\ln h(X+Y)]+\ln\left(\int_{\R^{n}}h\right)-\ent[X]\label{eq:ent-app2}
\end{equation}
for any integrable function $h:\R^{n}\to[0,+\infty)$ which is positive
on $2K$ (note that the first term on the right hand side depends
only on values of $h$ on $2K$, whereas the second term can only
get smaller or stay the same when $h$ is restricted to $2K$; in
other words, replacing $h$ with $h\one_{2K}$ might only improve
the right hand side).

Observe that, by choosing $h$ constant on $2K$ (and zero otherwise),
one can recover (\ref{eq:MP}). In the remainder of this section,
we will choose different $h$ in order to establish the improvements
of (\ref{eq:MP}) claimed earlier.
\end{rem}

\begin{prop}
\label{prop:MP_psi}Suppose $K$ is a convex body centered at the
origin which is $\psi_{\alpha}$ with constant $b_{\alpha}$. Then,
for some universal constant $c>0$,
\[
\frac{|K\cap(-K)|}{|K|}\ge\frac{\exp\left(cb_{\alpha}^{-\alpha}n^{\alpha/2}\right)}{2^{n}}.
\]
\end{prop}

\begin{proof}
We begin by observing that both sides of (\ref{eq:ent-app1}) are
invariant under invertible linear transformations of $K,$ therefore
we can assume without loss of generality that $K$ is in isotropic
position. We then apply (\ref{eq:ent-app2}) with $h(x):=\exp(-\lambda\|x\|_{2}^{2})\one_{2K}$
for some constant $\lambda$ to be specified later. The right hand
side becomes

\begin{align}
\EE[\lambda\norm{X+Y}_{2}^{2}]+\ln\int_{2K}\exp(-\lambda\|x\|_{2}^{2})\,dx-\ln1 & =2\EE[\lambda\|X\|_{2}^{2}]+\ln\int_{2K}\exp(-\lambda\|x\|_{2}^{2})\,dx\nonumber \\
 & =2\lambda nL_{K}^{2}+n\ln2+\ln\int_{K}\exp(-4\lambda\|x\|_{2}^{2})\,dx.\label{eq:psi-ent-bound1}
\end{align}

To estimate the last integral, we employ again the thin-shell estimates
from Theorem \ref{thm:GM2011}, which imply that for $A_{t}:=\left\{ x\in K:\norm x_{2}\leq(1-t)\sqrt{n}L_{K}\right\} $,
one has 
\[
\left|A_{t}\right|\leq C\exp\left(-c^{\prime}b_{\alpha}^{-\alpha}t^{2+\alpha}n^{\alpha/2}\right)
\]
for all $t\in\left[0,1\right].$ We can thus break the integral into
two as follows:

\begin{align*}
\int_{K}\exp(-4\lambda\|x\|_{2}^{2})\,dx & =\int_{A_{t}}\exp(-4\lambda\|x\|_{2}^{2})\,dx+\int_{K\setminus A_{t}}\exp(-4\lambda\|x\|_{2}^{2})\,dx\\
 & \leq C\exp\left(-c^{\prime}b_{\alpha}^{-\alpha}t^{2+\alpha}n^{\alpha/2}\right)+\exp(-4\lambda(1-t)^{2}nL_{K}^{2}).
\end{align*}
 We now set $t=1-\frac{2}{\sqrt{5}}$ say, and then we choose our
$\lambda$ so that
\[
c^{\prime}b_{\alpha}^{-\alpha}t^{2+\alpha}n^{\alpha/2}=4\lambda(1-t)^{2}nL_{K}^{2}.
\]
 It follows that $\lambda$ is of the order of $b_{\alpha}^{-\alpha}n^{\alpha/2-1}L_{K}^{-2}$.
Combining these estimates with (\ref{eq:ent-app2}) and (\ref{eq:psi-ent-bound1}),
we obtain
\begin{align*}
-\ln\left(\frac{|K\cap(-K)|}{|K|}\right) & \leq2\lambda nL_{K}^{2}+n\ln2+\ln(C+1)-\frac{16}{5}\lambda nL_{K}^{2}\\
 & =n\ln2+\ln(C+1)-\frac{6}{5}\lambda nL_{K}^{2}\\
 & =n\ln2+\ln(C+1)-c^{\prime\prime}b_{\alpha}^{-\alpha}n^{\alpha/2}
\end{align*}
for some absolute constant $c^{\prime\prime}$ (which we can compute
explicitly by the above relations). Exponentiating, we complete the
proof.
\end{proof}
\begin{proof}[Proof of second part of Theorem \ref{thm:uniform_convex}]
 This time we only assume for simplicity that $|K|=1,$ and we apply
(\ref{eq:ent-app2}) with $h(x):=\exp\left(-\lambda\|x\|_{K}\right)$
for some constant $\lambda$ to be specified later. We immediately
get 

\begin{align*}
-\ln\left(\frac{|K\cap(-K)|}{|K|}\right) & \leq\EE\left[\lambda\|X+Y\|_{K}\right]+\ln\irn\exp(-\lambda\|x\|_{K})\,dx\\
 & =\lambda\EE\left[\|X+Y\|_{K}\right]+\ln(\lambda^{-n}n!|K|)=\lambda\EE\left[\|X+Y\|_{K}\right]-n\ln\lambda+\ln(n!).
\end{align*}
 Optimizing over $\lambda$ yields
\begin{equation}
-\ln\left(\frac{|K\cap(-K)|}{|K|}\right)\leq n\ln\EE\left[\|X+Y\|_{K}\right]+\ln\frac{n!e^{n}}{n^{n}}.\label{eq:unif-ent-bound1}
\end{equation}
Given that $n!\leq en^{n+1/2}e^{-n},$ the last term is upper-bounded
by $\ln(e\sqrt{{n}})$, so the final estimate will depend on how
well we can bound $\EE\left[\|X+Y\|_{K}\right].$ We will use again
the concentration result of Arias-De-Reyna, Ball, and Villa. Note
that by the triangle inequality $\|X+Y\|_{K}\leq2$, and therefore,
by the definition of the modulus of convexity, we have for any $\varepsilon\in(0,2),$

\begin{align*}
\EE\left[\|X+Y\|_{K}\right] & =\EE\left[\|X+Y\|_{K}\one_{{\|X-Y\|_{K}\leq\varepsilon}}\right]+\EE\left[\|X+Y\|_{K}\one_{{\|X-Y\|_{K}>\varepsilon}}\right]\\
 & \leq2\PP\left(\|X-Y\|_{K}\leq\varepsilon\right)+2\left(1-\delta_{K}(\varepsilon)\right)\PP\left(\|X-Y\|_{K}>\varepsilon\right)\\
 & =2\left[1-\delta_{K}(\varepsilon)\PP\left(\|X-Y\|_{K}>\varepsilon\right)\right].
\end{align*}
 Applying this now with some $\varepsilon\in\left(0,\sqrt{2}\right)$
for which $\delta_{K}(\varepsilon)\geq r$, and recalling (\ref{eq:conc_dist}),
we obtain
\[
\EE\left[\|X+Y\|_{K}\right]\leq2\left[1-\delta_{K}(\varepsilon)\left(1-\exp\left(-\tfrac{\left(\sqrt{2}-\eps\right)^{2}n}{4}\right)\right)\right]\leq2\left[1-r\left(1-\exp\left(-\tfrac{\left(\sqrt{2}-\eps\right)^{2}n}{4}\right)\right)\right],
\]
 which we can plug into (\ref{eq:unif-ent-bound1}) to complete the
proof.
\end{proof}

\section{\label{sec:Conclusions}Concluding remarks}

We conclude this note with some remarks, questions and conjectures.
\begin{conjecture}
There exists a universal constant $c>0$ such that for every centered
convex body $K\sub\R^{n}$ and some $0<r<1$ one has 
\[
\frac{\PP\left(\frac{X+Y}{2}\in rK\right)}{\PP\left(X\in rK\right)}\ge\left(1+c\right)^{n},
\]
where $X$ and $Y$ are independent random vectors, uniformly distributed
on $K$. 
\end{conjecture}

We remark that the above conjecture implies an exponentially better
upper bound for Hadwiger's covering problem. Moreover, the conjecture
seems interesting in its own right and attempts in a way to quantify the intuition
that the convolution of a uniform distribution with itself looks already
more like a ``bell curve'' than like the flat distribution it originates
from. 

\smallskip

Another question that would capture this if answered in the
affirmative is the following. Let $X$ and $Y$ be independent random vectors, uniformly distributed
on a centred convex body $K$. Is it true that $\EE\|X+Y\|_{K} \leq 2 - \Omega(n^{-\alpha})$ with $\alpha\in [0,1)$ independent of $K$? Or rather, given such an $\alpha$, for which convex bodies in ${\mathbb R}^n$ does this bound hold? (Note that any such bound would improve on the trivial upper bound coming from the triangle inequality, which totally neglects independence: $\EE\|X+Y\|_{K} \leq 2\EE\|X\|_K = 2(1- \frac{1}{n+1})$.)  In the previous section we proved that $\EE\|X+Y\|_{K}$ is indeed upper-bounded by a constant smaller than $2$ for convex bodies with a positive modulus of convexity. For the cube however, it can be checked that $\EE\|X+Y\|_{K} = 2\bigl(1-\frac{4^n}{(2n+1)\binom{2n}{n}}\bigr) \sim 2 - \sqrt{\frac{\pi}{n}}$. Thus, we can also ask whether, in general, the bound $2-\Omega(n^{-1/2})$ is the worst case. If this is true, it would give another proof for our Main Theorems \ref{thm:Hadwiger} and \ref{thm:main}.

\medskip

The quantity $\ent\left[X+Y\right]-\ent\left[X\right]$ which appears
on the right hand side of (\ref{eq:ent-app1}) has been studied in
the context of reverse entropy power inequalities for convex measures,
a natural generalisation of log-concave measures (see \cite{BobMad13}
and \cite{MadKont18}). 
The upper bounds obtained there (when specialized
to the log-concave case) as well as our improved bounds are perhaps
far from optimal. To the best of our knowledge, a sharp upper bound
is not known even in dimension one. We believe the extremiser would
be a one-sided exponential distribution. 

Furthermore, in higher dimensions we can conjecture the following: for some universal constant $\varepsilon > 0$ and for every i.i.d. log-concave random vectors $X$ and $Y$ in $\R^n$, $\ent\left[X+Y\right]-\ent\left[X\right] \leq n(\ln 2 - \varepsilon)$. An even more ambitious guess here is that the extremiser should be the product one sided exponential distribution (giving the upper bound $n\gamma$ with $\gamma = 0.57\ldots$ denoting the Euler-Mascheroni constant) and the simplex for uniform distributions on convex bodies. 

Recall that our strategy from the proof of Proposition \ref{prop:MP_psi} for bounding $\ent\left[X+Y\right]-\ent\left[X\right]$ was to normalise $X$ to be isotropic and choose a Gaussian function $h(x) = \exp(-\lambda\|x\|_{2}^{2})$ in \eqref{eq:ent-app1}. Note however that any further improvements while working with this function $h$ might be particularly hard: by relying on now classical volume concentration results as well as on reductions for the slicing problem from \cite{dafnis}, it is possible to check that, if this choice for $h$ yields the bound $n(\ln 2 - \varepsilon)$ for uniform random vectors, then this also implies logarithmic bounds for the slicing problem.

\medskip

Theorem \ref{thm:main} has an immediate application in the geometry of numbers, and in particular to Ehrhart's conjecture from \cite{Ehr1, Ehr2}. This conjecture states that for every convex body $K$ in $\R^n$ with barycenter at the origin and such that the only lattice point of $\mathbb{Z}^n$ in the interior of $K$ is the origin, we have $|K| \leq \frac{(n+1)^n}{n!}$ (with equality attained when $K$ is the simplex $K = (n+1)\text{conv}\{0, e_1, \dots, e_n\}-(1,\dots,1)$).

Ehrhart's conjecture has been confirmed for $n=2$ by Ehrhart, and in some special cases (see \cite{BeBe, Ehr3, NiPa}), but it remains open in general for $n \geq 3$. The general bound $|K| \leq (n+1)^n(1-(1-1/n)^n)$ was established in \cite{GrWi}. The better bound $|K|\le 4^{n}$ is a direct consequence of a more general result, namely \cite[Proposition 1.1]{BeHe}, concerning a strengthening of Ehrhart's conjecture; see also \cite{HeHeHe} for a simpler derivation of the bound $|K|\le 4^{n}$, which we also follow below.  

Both derivations of this bound made use of the Milman-Pajor inequality \eqref{eq:MP}, so Theorem \ref{thm:main} allows us now to obtain the following improvement.

\begin{prop}\label{prop:ehr}
Let $K$ be a convex body in $\R^n$ with $b(K)=0$ such that $\text{int}(K) \cap \mathbb{Z}^n = \{0\}$. Then $|K| \leq 4^ne^{-c\sqrt{n}}$, where $c > 0$ is a universal constant.
\end{prop}
\begin{proof}
Since $K \cap (-K)$ is origin-symmetric and its interior contains no lattice point other than the origin, by Minkowski's theorem we obtain that $|K \cap (-K)| \leq 2^n$. Thus, Theorem \ref{thm:main} gives $\frac{2^n}{|K|} \geq 2^{-n}e^{c\sqrt{n}}$.
\end{proof}

\begin{acknowledgements*}
We are indebted to Martin Henk for remarking on the applicability of our main result to Ehrhart's conjecture. We also thank Shiri Artstein-Avidan and Bo'az Klartag for useful discussions. 

We thank the Mathematical Sciences Research Institute in Berkeley,
California, where part of this work was done during the Fall 2017
semester; we acknowledge the support of the institute and of the National
Science Foundation under Grant No. 1440140. TT was supported in part by the Collaboration Grants from the Simons Foundation and NSF grant DMS-1955175.
\end{acknowledgements*}
\bibliographystyle{amsplain_abr}
\bibliography{refs}

\end{document}